\documentclass[12pt]{amsart}
\usepackage[T1]{fontenc}
\usepackage[utf8]{inputenc}
\usepackage{color,amssymb,amsmath,mathrsfs,enumerate,esint}
\usepackage{a4wide}
\usepackage{hyperref}
\theoremstyle{plain}
\newtheorem{Theorem}{Theorem}[section]
\newtheorem{Lemma}[Theorem]{Lemma}

\theoremstyle{definition}
\newtheorem{Remark}[Theorem]{Remark}
\newtheorem{Definition}[Theorem]{Definition}

\DeclareMathOperator*{\sgn}{sgn}

\DeclareMathOperator*{\dist}{dist}

\title{Interpolation between H\" older and Lebesgue spaces with applications}
\author{Anastasia Molchanova}
\address{Institute of Mathematics, Acad. Koptyug avenue 4, Novosibirsk, Russia}
\email{a.molchanova@math.nsc.ru}

\author{Tom\' a\v s Roskovec}
\thanks{The second author was supported by EF--IGS2017--Rost-IGS04B1.}
\address{Faculty of Economics, University of South Bohemia, Studentsk\' a 13, \v Cesk\' e Bud\v ejovice, Czech Republic}
\email{troskovec@ef.jcu.cz}

\author{Filip Soudsk\' y}
\thanks{The third author was supported by EF--IGS2017--Soudsk\' y--IGS07P1}
\address{Faculty of Economics, University of South Bohemia, Studentsk\' a 13, \v Cesk\' e Bud\v ejovice, Czech Republic}
\email{fsoudsky@ef.jcu.cz}

\begin{document}

\maketitle
{\centering\footnotesize In memory of V\'aclav N\'ydl.\par}
\begin{abstract}
Classical interpolation inequality of the type $\|u\|_{X}\leq C\|u\|_{Y}^{\theta}\|u\|_{Z}^{1-\theta}$ is well known in the case when $X$, $Y$, $Z$ are Lebesgue spaces. In this paper we show that this result may be extended by replacing norms $\|\cdot\|_{Y}$ or $\|\cdot\|_{X}$ by suitable H\" older semi-norm. We shall even prove sharper version involving weak Lorentz norm. We apply this result to prove the Gagliardo--Nirenberg  inequality for a wider scale of parameters.
\end{abstract}

\section{Introduction and main result}
The classical Sobolev embedding theorem claims that if $1\leq p<n$ then for any weakly differentiable function $u\in WL^{p}$ one has
$$
\|u\|_{p^{*}}\leq C\|\nabla u\|_{p}, 
$$
where $p^{*}=\frac{np}{n-p}$ and $C>0$ is independent of $u$. If $p>n$ then by the Morrey lemma, for the continuous representative the following holds
$$
\|u\|_{\mathcal{C}^{0,1-\frac{n}{p}}}\leq C \|\nabla u\|_{p}.
$$
These are classical results which can be found, for instance, in classical books
% the book of Adams (see 
\cite{AD} or \cite{EG}. Following the notation of L.~Nirenberg \cite[Lecture II]{Ni}, consider the extended norm for \\
$-\infty < \frac{1}{p}<\infty$.
\begin{Definition}\label{def:extended_def}
For $p>0$ the norm $\|\cdot\|_p$ is defined by
$$\|u\|_{p} = \left(\int_{\mathbb{R}^{n}} |u|^{p} dx\right)^{\frac{1}{p}};$$
for $p<0$ set numbers $s$ and $\tilde{p}$ by $s = [-n/p]$ (where $[\alpha]$ stands for the integer part of $\alpha$), $n/\tilde{p} = s + n/p$, and define
\begin{equation}
\begin{aligned}\label{HoldersemiGen}
& \|u\|_p = \|\nabla^s u\|_{\tilde{p}},  & \qquad \text{if} \: -\infty < \tilde{p} < -n, \\
& \|u\|_p = \|\nabla^s u\|_{\infty},  & \qquad \text{if} \: s = -n/p, \\
\end{aligned}
\end{equation}
where 
%supremum is taken with respect to all derivatives $D^s$ of order $s$, and 
the semi-norm $\|\cdot\|_{\tilde{p}}$, $-\infty < \tilde{p} < -n$, is defined by
\begin{equation}\label{HoldersemiEq}
\|u\|_{\tilde{p}}:=\displaystyle{\sup_{x,y\in\mathbb{R}^{n}}}\frac{|u(x)-u(y)|}{|x-y|^{-\frac{n}{\tilde{p}}}}.
\end{equation}
\end{Definition}
Equipped with this extended definition we may combine the Sobolev embedding with the Morrey lemma to obtain
$$
WL^{p}\hookrightarrow L^{p^{*}},
$$
for all $p\in(-\infty,0)\cup [1,\infty]\setminus\{n\}$. This approach was taken by L.~Nirenberg in his celebrated paper \cite{Ni}. In this paper an interpolation inequality
$$
\|u\|_{p}\leq C\|u\|_{r}^{\theta}\|u\|_{q}^{1-\theta}
$$
in the sense of the extended definition of $\|\cdot\|_{p}$ is needed in the proof of the Gagliardo--Nirenberg inequality. 
% The reference to a paper involving this key result is missing there, instead the conference that happened a few years earlier is claimed to have presented it. 
The problem was studied by A.~Kufner and A.~Wannabe (see \cite{Kuf}). However, the authors present only the proof for the case of two Lebesgue and one H\" older space. In our paper we shall fill the missing case and give sharper results proving estimates using weak Lorentz space instead of Lebesgue space in the upper bound. Let us remind the reader that if $1\leq p< \infty$ the norm in the space $L^{p,\infty}$ is defined by
$$
\|u\|_{p,\infty}:=\displaystyle{\sup_{t>0}}\,t|\{|u|>t\}|^{\frac{1}{p}}.
$$
One may observe that for any $p\geq 1$ the following holds
$$
\|u\|_{p,\infty}\leq \|u\|_{p}.
$$
We shall now present our main result. More precisely we show that classical interpolation holds even if we use our extended definition of $\|\cdot\|_{p}$. After the proof of the main result, we use our result to proof the most general form of the Gagliardo--Nirenberg interpolation inequality. 
In the particular case the inequality was introduced by O.~Ladyzhenskaya in \cite{La} to investigate the Navier--Stokes equations in two spatial dimensions.
More general version of the inequality was proved by E.~Gagliardo in \cite{Ga} and independently shortly after by L.~Nirenberg in \cite{Ni}.

The usage of the Gagliardo--Nirenberg inequality is enormous and exceeds the field of PDEs. We may mention the weak inverse mapping theorem \cite{CHK}. Finer versions and special cases of the inequality are still needed and formulated, see \cite{CFK} and \cite{MRR}.
\begin{Theorem}\label{MT}
Let $q\in[1,\infty]$, $p,r\in(-\infty,-n]\cup[1,\infty]$ and $\theta\in(0,1)$ be numbers for which
\begin{equation}
\frac{1}{p}=\frac{\theta}{r}+\frac{1-\theta}{q}.
\end{equation}
Then there exists a constant $C$ independent of $u$ such that 
\begin{equation}\label{InterpolationIneq}
\|u\|_{p}\leq C\|u\|_{r}^{\theta}\|u\|_{q,\infty}^{1-\theta}.
\end{equation}
\end{Theorem}

\section{Proofs}
In the proof for the quantities dependent on function $u$, $A=A(u),B=B(u)$ we shall write
$$
A\lesssim B
$$
if there exists constant $C>0$ independent of $u$ such that
$$
A\leq CB.
$$
We write $A\approx B$ if $A\lesssim B$ and $B\lesssim A$.

In the proof, we shall need one auxiliary isoperimetric inequality. To state it we need the following notation. For a set $M\subset\mathbb{R}^{n}$ and $t>0$ define
$$
\begin{aligned}
(M)_{t}&:=\{x\in\mathbb{R}^{n}:\dist(x,M^{c})>t\},\\
(M)^{t}&:=\{x\in\mathbb{R}^{n}:\dist(x,M)<t\}.
\end{aligned}
$$

\begin{Lemma}\label{BMR}
Let $M\subset \mathbb{R}^{n}$ be a Lebesgue measurable set of finite measure. Let $B$ be a ball such that $|B|=|M|$. Then for any $t>0$ one has $|(B)_{t}|\geq|(M)_{t}|$.
\end{Lemma}
\begin{proof}
Note that for arbitrary measurable set $M\subset \mathbb{R}^{n}$ one has $\left((M)_{t}\right)^{t}\subset M$, in particular $\left((B)_{t}\right)^{t}= B$ for any ball $B$. Let $B$ be a ball satisfying $|B|=|M|$. Assume that there exists $t>0$ for which $|(B)_{t}|<|(M)_{t}|$ holds. Then there exists $\delta>0$ such that $|((B)^\delta)_{t}|=|M_{t}|$. By a version of the Brun--Minkowski inequality (see \cite{II}[Theorem III.2.2]) we have
$$
|B|<|(B)^{\delta}|\leq|((M)_{t})^{t}|\leq |M|,
$$
which is a contradiction. Therefore $|(M)_{t}|\leq|(B)_{t}|$ for all $t>0$.
\end{proof}

\begin{proof}[Proof of Theorem \ref{MT}]
The classical case with $p$, $q$, $r\in[1,\infty]$ is a folklore. Let us focus on the other cases. First suppose that $p$, $q\in[1,\infty]$ and $r\in(-\infty,-n]$.

\par

First note that we may assume that $u\in L^{q,\infty}\cap \mathcal{C}^{0,-\frac{n}{r}}$, otherwise there is nothing to prove. In particular, $u$ is continuous. For given $s>0$, let us define
$$
u_{s}:=\sgn(u)\min\{|u|,s\}.
$$
Observe 
\begin{equation}\label{splitcase1Eq}
\|u\|_{p}^{p}\lesssim\|u-u_{s}\|_{p}^{p}+\|u_{s}\|_{p}^{p}.
\end{equation}
Let us estimate each term separately. First denote
$$
E_{s}:=\{|u|>s\}.
$$
By \eqref{HoldersemiEq} and the Fubini theorem we estimate
\begin{equation}\label{uminususEq}
\begin{aligned}
\|u-u_{s}\|_{p}^{p}&=\int_{E_{s}}(|u|-s)^{p}\textup{d}x\\
&\leq\|u\|_{r}^{p}\int_{E_{s}}\dist(x,E_{s}^{c})^{-\frac{pn}{r}}\textup{d}x\\
&\leq\|u\|_{r}^{p}\int_{0}^{\infty}|\{x:\dist(x,E_{s}^{c})>t^{-\frac{r}{pn}}\}|\textup{d}t.
\end{aligned}
\end{equation}
Let us consider ball $B=B(0,\rho)$, such that $|B|=|E_{s}|$ (therefore $\rho=\omega_{n}^{-\frac{1}{n}}|E_{s}|^{\frac{1}{n}}$). By Lemma~\ref{BMR} we obtain the estimate for the size of the set $$|\{x:\dist(x,E_{s}^{c})>t^{-\frac{r}{pn}}\}|\leq |B(0,\rho-t^{-\frac{r}{pn}})|,$$ hence the last expression of \eqref{uminususEq} can be further estimated by
$$
\|u\|_{r}^{p}\int_{0}^{\rho^{-\frac{pn}{r}}}|B(0,\rho-t^{-\frac{r}{pn}})|\textup{d}t\approx\|u\|_{r}^{p}\rho^{n-\frac{pn}{r}}\approx\|u\|_{r}^{p}|\{|u|>s\}|^{1-\frac{p}{r}}.
$$

\par

Let us now estimate the other term form \eqref{splitcase1Eq}. We obtain
$$
\begin{aligned}
\|u_{s}\|_{p}^{p}&=\int_{\mathbb{R}^{n}}\min\{s,|u|\}^{p}\textup{d}x\\
&=\int_{0}^{s^{p}}|\{|u|>z^{\frac{1}{p}}\}|\textup{d}z\\
&\approx\int_{0}^{s}t^{p-1}|\{|u|>t\}|\textup{d}t\\
&\lesssim\int_{0}^{s}t^{p-q-1} t^{q}|\{|u|>t\}|\textup{d}t\\
&\lesssim\|u\|_{q,\infty}^{q}\int_{0}^{s}t^{p-q-1}\textup{d}t    \\
&\lesssim s^{p-q}\|u\|_{q,\infty}^{q} .
\end{aligned}
$$
Therefore
$$
\begin{aligned}
\|u\|_{p}^{p}&\lesssim\|u-u_{s}\|_{r}^{p}|\{|u|>s\}|^{1-\frac{p}{r}}+s^{p-q}\|u_{s}\|_{q,\infty}^{q}\\
&\leq \|u\|_{r}^{p}|\{|u|>s\}|^{1-\frac{p}{r}}+s^{p-q}\|u\|_{q,\infty}^{q}.
\end{aligned}
$$
Now we may suppose that the distribution function is continuous (otherwise we approximate $u$ by functions with continuous distribution function in $L^{q,\infty}\cap L^{r}$). Therefore there exists $s\in(0,\infty)$ such that 
$$
\frac{\|u\|_{r}^{p}}{\|u\|_{q,\infty}^{q}}=s^{p-q}|\{|u|>s\}|^{\frac{p}{r}-1}.
$$
For this choice of $s$ both term on the right-hand side of the estimate are equal. We obtain
$$
\begin{aligned}
\|u\|_{p}^{p}&\lesssim\left(\|u\|_{r}^{p}|\{|u|>s\}|^{1-\frac{p}{r}}\right)^{\theta}\left(s^{p-q}\|u\|_{q,\infty}^{q}\right)^{1-\theta}\\
&\leq\|u\|_{r}^{p\theta}|\{|u|>s\}|^{\theta-\frac{\theta p}{r}}s^{(1-\theta)(p-q)}\|u\|_{q,\infty}^{q(1-\theta)}\\
&\leq \|u\|_{r}^{p\theta}\|u\|_{q,\infty}^{(p-q)(1-\theta)}\|u\|_{q,\infty}^{q(1-\theta)}\\
&\leq \|u\|_{r}^{p\theta}\|u\|_{q,\infty}^{p(1-\theta)}.
\end{aligned}
$$
Rise both side of the inequality to $\frac{1}{p}$ to finish the proof.

\par
\bigskip

Let us consider the other case, that is $p$, $r\in(-\infty,-n]$ and $q\in[1,\infty)$. We may suppose that $u\in L^{r}\cap L^{q,\infty}$ otherwise there is nothing to prove. We claim that if $|u(x)|>0$ then
$$
B\left(x,\left(\frac{|u(x)|}{2}\right)^{-\frac{r}{n}}\|u\|_{r}^{\frac{r}{n}}\right)\subset\left\{|u|>\frac{|u(x)|}{2}\right\}.
$$
To prove the claim suppose $y$ is a point which is not included in the set on the right-hand side. We have $|u(y)|\leq\tfrac{|u(x)|}{2}$, hence, using \eqref{HoldersemiEq}, we obtain
$$
\|u\|_{r}|x-y|^{-\frac{n}{r}}\geq|u(x)-u(y)|\geq \frac{|u(x)|}{2}.
$$
That implies that $|y-x|$ is bigger than the radius and $y$ does not belong to the left hand side set.
%which implies 
%$$
%|x-y|\geq \left(\frac{|u(x)|}{2}\right)^{-\frac{r}{n}}\|u\|_{r}^{\frac{r}{n}}
%$$

We measure the sets used in the claim to get
$$
\left(\frac{|u(x)|}{2}\right)^{-r+q}\|u\|_{r}^{r}\omega_{n}\leq\left(\frac{|u(x)|}{2}\right)^{q}\left|\left\{|u|>\frac{|u(x)|}{2}\right\}\right|.
$$
Hence, the point-wise estimate
\begin{equation}\label{PWE}
|u(x)|\lesssim\|u\|_{q,\infty}^{\frac{q}{q-r}}\|u\|_{r}^{\frac{r}{r-q}}
\end{equation}
holds. Equipped by this estimate we may proceed to the proof. Let's consider sets
$$
N_{s}:=\{(x,y)\in\mathbb{R}^{2n}:|x-y|\leq s\}
$$
and
$$
D_{s}:=\{(x,y)\in\mathbb{R}^{2n}:|x-y|>s\}.
$$
Now we have
$$
\|u\|_{p}\leq\displaystyle{\sup_{(x,y)\in N_{s}}}\frac{|u(x)-u(y)|}{|x-y|^{-\frac{n}{p}}}+\displaystyle{\sup_{(x,y)\in D_{s}}}\frac{|u(x)-u(y)|}{|x-y|^{-\frac{n}{p}}}=:I+II.
$$
Using \eqref{HoldersemiEq}, the first term may be easily estimated by
$$
I\leq\|u\|_{r}s^{-\frac{n}{r}+\frac{n}{p}}.
$$
As for the other one, we just realize that
$$
II\leq 2\max\{|u(x)|,|u(y)|\}s^{\frac{n}{p}}\lesssim s^{\frac{n}{p}}\|u\|_{q,\infty}^{\frac{q}{q-r}}\|u\|_{r}^{\frac{r}{r-q}}.
$$
Therefore
$$
I+II\lesssim s^{\frac{n}{p}}\|u\|_{q,\infty}^{\frac{q}{q-r}}\|u\|_{r}^{\frac{r}{r-q}}+\|u\|_{r}s^{-\frac{n}{r}+\frac{n}{p}}.
$$
Choose
$$s=\left(\frac{\|u\|_{q,\infty}}{\|u\|_r}\right)^{-\frac{qr}{n(q-r)}},$$
to finish the proof.
\end{proof}

\section{Applications}

Let us show how our result~\eqref{MT} can be used to prove the most general version of the Gagliardo--Nirenberg interpolation inequality once we have proved it for the least possible value of parameter $\theta$. We also noticed that the Gagliardo--Nirenberg interpolation inequality doesn't hold for all the values of parameters for which Nirenberg stated it (see \cite{Ni}). Let us make a little correction and exclude some special cases for which the theorem doesn't hold. For this purpose define inductively the following notation
$$
\begin{aligned}
r^{(0)}&:=r,\\
r^{(k)}&:=\left(r^{(k-1)}\right)^{*}\quad \textup{for }k\in\mathbb{N}. 
\end{aligned}
$$  
%\begin{Theorem}\label{GN-corrected}
%The inequality
%\begin{equation}\label{Gagliardo--Nirenberg}
%\|\nabla^{j} u \|_{p} \leq C \|\nabla^k u\|^{\theta}_{r} %\|u\|^{1-\theta}_{q}
%\end{equation}
%holds for any $j,k\in\mathbb{N}$, $\theta \in [\frac{j}%{k},1]$,
%where $q \in [1,\infty]$, $p\in(-\infty,0)\cup(1,\infty]$, $r %\in (-\infty,0) \cup [1,\infty]$, are numbers satisfying
%$1\leq j < k$ and
%\begin{equation}
%\frac{1}{p} = \frac{j}{n} + \theta \left(\frac{1}{r} - %\frac{k}{n} \right) + \frac{1-\theta}{q},
%\end{equation}
%moreover 
%\begin{equation}\label{UP}
%p^{(i)}\neq n
%\end{equation}
%for all $0\leq i\leq k-j-1$. 
%\end{Theorem}
\begin{Theorem}\label{GN-corrected}
Let $j$, $k\in\mathbb{N}$, $\theta \in [\frac{j}{k},1]$, $q \in [1,\infty]$, $p\in(-\infty,0)\cup(1,\infty]$, \\$r \in (-\infty,0) \cup [1,\infty]$ be numbers satisfying
$1\leq j < k$ and
\begin{equation}
\frac{1}{p} = \frac{j}{n} + \theta \left(\frac{1}{r} - \frac{k}{n} \right) + \frac{1-\theta}{q},
\end{equation}
moreover 
\begin{equation}\label{UP}
r^{(i)}\neq n
\end{equation}
for all $0\leq i\leq k-j-1$. 
Then there exists $C>0$ such that inequality
\begin{equation}\label{Gagliardo--Nirenberg}
\|\nabla^{j} u \|_{p} \leq C \|\nabla^k u\|^{\theta}_{r} \|u\|^{1-\theta}_{q}
\end{equation}
holds.
\end{Theorem}
We shall see from the proof that condition \eqref{UP} is essential and cannot be omitted since the embedding
$$
VL^{n}\hookrightarrow L^{\infty}
$$
does not hold. The proof is done by interpolation of the Gagliardo--Nirenberg inequality in the case of $\theta=j/k$ and the iterated Sobolev embedding
\begin{equation}\label{SOB}
V^{k-j}L^{r}\hookrightarrow L^{r^{(k-j)}}.
\end{equation}
which holds for $r>1$ if and only if \eqref{UP} is satisfied. Let us proceed with the proof.
\begin{proof}
The inequality~\eqref{Gagliardo--Nirenberg} is fulfilled for the extreme values of $\theta$, $\theta = j/k$ and $\theta = 1$. For original proof see \cite{Ni}. More reader-friendly and modern proof,  based on a pointwise estimate in terms of the Hardy--Littlewood maximal function, may be found in \cite{MS}.
% Note here that the case $\theta = 1$ correspond to the Sobolev embedding theorem if $1 \leq r < n$ and the Morrey lemma if $r > n$. 

Indeed, define $p(\theta)$ by 
\begin{equation*}
\frac{1}{p(\theta)}= \frac{j}{n} + \theta \left(\frac{1}{r} - \frac{k}{n} \right) + \frac{1-\theta}{q}.
\end{equation*}
Since the function $\left(p(\theta)\right)^{-1}$ is a monotone function of $\theta$, there exists $\zeta\in(0,1)$ such that 
$$
\frac{1}{p(\theta)}=\frac{\zeta}{p(\frac{j}{k})}+\frac{1-\zeta}{p(1)}.
$$
From where we derive that $\zeta$ satisfies $\theta = 1 - \zeta \left(1 - \frac{j}{k} \right)$.

If $q \in [1,\infty]$, $p$, $r \in (-\infty,-n) \cup [1,\infty]$,
from~\eqref{InterpolationIneq},\eqref{SOB} and~\eqref{Gagliardo--Nirenberg} 
we derive
$$
\begin{aligned}
\|\nabla^{j} u\|_{p(\theta)} & \lesssim\|\nabla^{j} u\|_{p(\frac{j}{k})}^{\zeta}\|\nabla^{j} u\|_{p(1),\infty}^{1-\zeta} \\
& \lesssim  \|\nabla^{j} u\|_{p(\frac{j}{k})}^{\zeta}\|\nabla^{j} u\|_{p(1)}^{1-\zeta}\\
& \lesssim \|\nabla^k u\|_{r}^{1 - \zeta(1 - \frac{j}{k})} \|u\|_{q}^{\zeta(1-\frac{j}{k})}.
\end{aligned}
$$

%Indeed
%$$\frac{j}{n} + \theta \left(\frac{1}{r} - \frac{k}{n} \right) + \frac{1-\theta}{q} = \zeta \left(\frac{j}{n} + \frac{j}{k} \left(\frac{1}{r} - \frac{k}{n} \right) + \frac{1-\frac{j}{k}}{q}\right) + (1-\zeta)\left(\frac{j}{n} + \frac{1}{r} - \frac{k}{n}\right)$$
%$$ \left(\frac{1}{r} - \frac{k}{n} \right) \left(\theta - 1 + \zeta \left(\frac{j}{k} - 1\right)\right)
% + \frac{1}{q} \left(1-\theta - \zeta \left(1 - \frac{j}{k}\right) \right) = 0.$$

If $q \in [1,\infty]$, $p \in (-\infty,-n) \cup [1,\infty]$, $r \in [-n,0)$, by Definition~\ref{def:extended_def} the inequality~\eqref{Gagliardo--Nirenberg} can be rewritten in the form
$$\|\nabla^{j} u \|_{p} \lesssim \|\nabla^{k+s_r} u\|^{\theta}_{\tilde{r}} \|u\|^{1-\theta}_{q},$$
where $s_r = [-n/r]$, $n/\tilde{r} = s_r + n/r$, $\tilde{r} \in (-\infty,-n)$ 
($\tilde{r} = \infty$ and $r^{-1}= 0$ if $s_r = -n/r$)
and the equality 
$$\frac{1}{p} = \frac{j}{n} + \theta \left(\frac{1}{\tilde{r}} - \frac{k+s_r}{n} \right) + \frac{1-\theta}{q}$$
is equivalent to 
$$\frac{1}{p} = \frac{j}{n} + \theta \left(\frac{1}{r} - \frac{k}{n} \right) + \frac{1-\theta}{q}.$$
The other cases are similar.

% If $q \in [1,\infty]$, $p \in [-n,0)$ and $r \in (-\infty,-n) \cup [1,\infty]$,
% corresponding $s_p = [-n/p]$, $n/\tilde{p} = s_p + n/p$, $\tilde{p} \in (-\infty,-n)$
% ($\tilde{p} = \infty$ and $\frac{1}{\tilde{p}} = 0$ if $s_p = -n/p$)
% and the equality 
% $$\frac{1}{\tilde{p}} = \frac{j+s_p}{n} + \theta \left(\frac{1}{r} - \frac{k}{n} \right) + \frac{1-\theta}{q}$$
% is equivalent to 
% $$\frac{1}{p} = \frac{j}{n} + \theta \left(\frac{1}{r} - \frac{k}{n} \right) + \frac{1-\theta}{q}.$$

% And if $q \in [1,\infty]$, $p \in [-n,0)$ and $r \in [-n,0)$,
% corresponding $s_p = [-n/p]$, $n/\tilde{p} = s_p + n/p$, $\tilde{p}r \in (-\infty,-n)$, ($\tilde{p} = \infty$ and $\frac{1}{\tilde{p}} = 0$ if $s_p = -n/p$)
% $s_r = [-n/r]$, $n/\tilde{r} = s_r + n/r$, $\tilde{r} \in (-\infty,-n)$, ($\tilde{r} = \infty$ and $\frac{1}{\tilde{r}} = 0$ if $s_r = -n/r$)
% and the equality 
% $$\frac{1}{\tilde{p}} = \frac{j+s_p}{n} + \theta \left(\frac{1}{\tilde{r}} - \frac{k+s_r}{n} \right) + \frac{1-\theta}{q}$$
% is equivalent to 
% $$\frac{1}{p} = \frac{j}{n} + \theta \left(\frac{1}{r} - \frac{k}{n} \right) + \frac{1-\theta}{q}.$$

\end{proof}

\begin{Remark}
In the critical cases (that is when $r^{(i)}=n$) the Gagliardo--Nirenberg inequality in the original L.~Nirenberg's statement cannot hold for $\theta=1$. In the case of $\theta\in(j/k, 1)$ the situation may be saved by using the space $\textup{BMO}$ instead of $L^{\infty}$ (for definition see \cite{PKJF}). We use the result of A.~Cianchi and L.~Pick (see \cite[Theorem 3.1.]{CP}) to obtain embedding 
$$W^{1}L^{n}\hookrightarrow \textup{BMO}.
$$
Now instead of using the interpolation inequality
$$
\|u\|_{p}\leq C\|u\|_{r}^{r/p}\|u\|_{\infty}^{1-r/p}
$$
we use the following sharper version 
$$
\|u\|_{p}\leq\|u\|_{r}^{r/p}\|u\|_{\textup{BMO}}^{1-r/p}.
$$
\end{Remark}
To prove the inequality for the remaining cases would require to prove this interpolation inequality in the upgraded version using the space BMO instead of $L^{\infty}$ when needed. In the case of  $p$, $r\geq 1$ the result is known. The remaining case we leave as an open problem. The BMO spaces can also be used in the Gagliardo--Nirenberg inequality instead of Lebesgue spaces, for details see \cite{MRR} or \cite{SP}.
\bibliographystyle{plain}

\end{document}